\documentclass[11pt]{amsart}
\usepackage{amsmath}
\usepackage[utf8]{inputenc}
\usepackage{amssymb}
\usepackage{amsopn}
\usepackage{epsfig}
\usepackage{amsfonts}
\usepackage{latexsym}
\usepackage{graphicx}
\usepackage{enumerate}
\usepackage{color}
\usepackage{tikz}
\newtheorem{theorem}{Theorem}[section]
\newtheorem{lemma}[theorem]{Lemma}

\newtheorem{corollary}[theorem]{Corollary}

\theoremstyle{definition}

\newtheorem{example}[theorem]{Example}

\newtheorem{conjecture}[theorem]{Conjecture}

\theoremstyle{remark}
\newtheorem{remark}[theorem]{Remark}
\newtheorem{problem}[theorem]{Problem}
\numberwithin{equation}{section}

\begin{document}
\title{Thickness theorems with an application in number theory}
\author{Kan Jiang}
\address[K. Jiang]{Department of Mathematics, Ningbo University,
People's Republic of China}
\email{jiangkan@nbu.edu.cn}
\date{\today}
\begin{abstract}
 In this paper, we  prove some new thickness theorems with partial derivatives.
We give some applications. First, we give a simple criterion that can judge whether two scaled Cantor sets have non-empty intersection. Second,
 we prove under  some checkable conditions that the continuous image of arbitrary self-similar sets with positive similarity ratios is a closed interval,   a finite union of closed intervals or  containing interior. Third, we prove an analogous Erd\H{o}s-Straus conjecture on the middle-third Cantor set.  Finally,  we consider the solutions to the Diophantine equations on fractal sets. More specifically,
for various Diophantine equations, we cannot find a solution on certain self-similar  sets, whilst for the Fermat's equation, which is associated with the famous Fermat's last theorem, we can find infinitely many  solutions  on many self-similar sets.
 \end{abstract}
\maketitle
\section{Introduction}
 Newhouse \cite{Palis} proved the following gap lemma.
 \begin{theorem}[\textbf{Newhouse's gap lemma}]
 Let $C_1$ and $C_2$ be  two  Cantor sets in $\mathbb{R}$,  if  neither set lies in a gap of the other, and $\tau(C_1)\cdot \tau(C_2)> 1$, then
 $C_1\cap C_2\neq \emptyset,$ where $\tau(C_i)$ denotes the thickness of $C_i, i=1,2.$
 \end{theorem}
 With a slight modification, we are allowed to prove the following result, which now is called the Newhouse's thickness theorem.
 \begin{theorem}[\textbf{Newhouse's thickness theorem}]
Let  $C_1$ and $C_2$ be two  linked Cantor sets, i.e. the size of largest gap of $C_1$ is not greater than the diameter of $C_2$ and vice versa. If $\tau(C_1)\cdot \tau(C_2)\geq 1$,
 then the arithmetic sum $$C_1+C_2=\{x+y:x\in C_1, y\in C_2\}$$ is an interval.
 \end{theorem}
 Astels \cite{Astels} generalized Newhouse's thickness theorem by considering multiple sum.
 Let $C_i$ be a Cantor set, $1\leq i\leq d$. We say  that $\{C_i\}_{i=1}^{d}$ are linked if
 \begin{itemize}
 \item  $|I_{r+1}|\geq |O_j|$ for $r=1,\cdots, d-1 \mbox{ and }j=1,\cdots, r$;
 \item  $|I_1|+\cdots+|I_r|\geq |O_{r+1}|$ for $r=1,\cdots, d-1$,
 \end{itemize}
 where $I_i$ stands for the convex hull of $C_i$, and $O_i$ refers to a gap with maximal size in $C_i.$ We write $|A|$ for the length or diameter of $A$.

 Astels proved the following result.
 \begin{theorem}\label{Astels}
 Let $\{C_i\}_{i=1}^{d}$ be  Cantor sets.
 \begin{itemize}
 \item If $$\sum_{i=1}^{d}\dfrac{\tau(C_i)}{\tau(C_i)+1}\geq 1,$$
 then  $$C_1+\cdots +C_d=\left\{\sum_{i=1}^{d}x_i:x_i\in C_i\right\} $$ contains some interiors.
 \item If $\{C_i\}_{i=1}^{d}$ are linked and   $\sum_{i=1}^{d}\dfrac{\tau(C_i)}{\tau(C_i)+1}\geq 1,$ then $C_1+\cdots +C_d$ is an interval.
 \end{itemize}
 \end{theorem}
Newhouse's and Astels' thickness theorems have many variants. We introduce some  results. Chronologically,  Moshchevitin \cite{Moshchevitin2000} proved an elegant thickness theorem, which generalizes  the Astels' thickness theorem.
However, Moshchevitin only gave a very simple outline of the proof. The next thickness theorem is due to Simon and Taylor \cite{ST}.
They proved, under some conditions, that the continuous image (the definition is in the next page) of  fractal sets contains interior.
 By virtue of their thickness theorem they proved some interesting results. Namely,  the sum of a two dimensional Cantor set and a circle contains some interior. Moreover, Simon and Taylor considered the pinned Falconer distance set, and obtained under some conditions that the pinned Falconer distance set contains some interior.
 The author considered under some conditions that the continuous image of Cantor sets is an interval \cite{Jiang2022}.

 Thickness theorems have many applications in number theory,  fractal geometry and dynamical systems. The reader may refer to \cite{Astels,  Feng2020, Iosevich2021, Iosevich2012, Taylor2021, Moshchevitin2000,Palis, ST, Takahashi}.
 We  introduce some applications in number theory.
 In fact,
various thickness theorems establish some relation between analytic number theory and fractal geometry\cite{Yu}. In analytic number theory, usually it is difficult to prove  some classical problems and  conjectures, for instance, the Waring problem \cite{Erdos1980}. However, with the help of thickness theorems, we can easily prove similar statements in fractal geometry. For instance, we can prove an analogous ``Lagrange's four-square theorem" on the classical middle-third Cantor set \cite{Wang,Yu}.
That is, let $x\in[0,4]$, then there exist some $x_i\in C, 1\leq i\leq 4$ such that
$$x=x_1^2+x_2^2+x_3^2+x_4^2,$$ where $C$ is the middle-third Cantor set.
In this paper, we give another application, i.e.
we prove an analogous  Erd\H{o}s-Straus conjecture on the middle-third Cantor set.  Moreover, we consider  whether some Diophantine equations have a solution on Cantor sets. In particular, we can find  the  fractal solutions to the Fermat's equation, which is associated with the famous Fermat's last theorem, on many self-similar sets. The reader may find similar applications in Section 3.
Thickness theorems have some applications in geometry.   Greenleaf,  Iosevich, and Taylor\cite{Iosevich2021} proved under some conditions that the configuration sets have some interiors, McDonald and Taylor\cite{Taylor2021} considered the  modified  distance problem, and obtained several interesting results.

Although
 the above mentioned thickness theorems are very useful to many fractal sets, all of them simultaneously  need two crucial  conditions, i.e.
Cantor sets  should  have large thicknesses, and they should  be linked (for different thickness theorems  the linked conditions differ).  Moreover, the thickness is difficult   to be estimated. For instance, it is very difficult to calculate the thicknesses of  the self-similar sets with complicated overlaps.
 For small thicknesses, i.e. thickness is strictly smaller than $1$, to the best of our knowledge,  there is no general thickness theorem, although we can find some examples such that the product of thicknesses is strictly smaller than $1$, and the sum of two specific Cantor sets is still an interval. For the linked condition, Takahashi\cite{Takahashi} proved that if
$\tau(K_1)\tau(K_2)\geq 1,$ then $K_1+K_2$ is a finite union of closed intervals. This result is still correct in a general setting.
The purpose  of this paper is to make a contribution to
 investigating  thickness theorems without utilizing thicknesses. Moreover, we prove an analogous Takahashi's thickness theorem without the linked condition.

Newhouse's and Astels' thickness theorems only investigate the sum of Cantor sets. Indeed, thickness theorems can be considered in  a very general setting, i.e.  the continuous images of fractal sets.
We introduce some definitions and results.
Let $A_i\subset \mathbb{R}, 1\leq i\leq d$ be a nonempty set, and  $f:\mathbb{R}^d\to \mathbb{R}$ be a $C^1$ function.
We define the continuous image of $f$ as follows:
$$f(A_1,\cdots, A_d)=\{f(x_1\cdots,x_d):x_i\in A_i, 1\leq i\leq d\}.$$
If each $A_i$ is an interval, then the topological structure of $f(A_1,\cdots, A_d)$ is clear. Nevertheless, if each $A_i$ is  a fractal set, then usually it is difficult to describe the Hausdorff dimension or topological structure of $f(A_1,\cdots, A_d)$. For instance, let $K_1$ and $K_2$ be two self-similar sets in $\mathbb{R}$, to the best of our knowledge, generally, we do not know when
$$K_1*K_2=\{x*y:x\in K_1, y\in K_2\},*=+,-,\cdot,\div $$
contains interior (if $*=\div$, then we assume $y\neq 0$).  In particular, given  two fractal sets $K_1$ and $K_2$,  we,  generally,   do not know whether
$$K_1+K_2=\{x+y:x\in K_1, y\in K_2\}$$ contains interior.  This problem is related to the celebrated Palis conjecture \cite{Yoccoz,Palis}.
The topological structure of $K_1*K_2$ is much more complicated. For instance, even for the middle-third Cantor set $C$, we do not completely  understand the topological structure of
$$C\cdot C=\{xy:x,y\in C\}. $$
The reader may refer to \cite{Gu}.

In fractal geometry, it is natural to ask when the continuous image of some fractal sets
 is an interval,  a union of finitely many closed intervals or containing some interior, provided that the continuous image of some fractal sets has full Hausdorff dimension, see \cite{BABA1,Hochman2012,PS}  and the references therein.    Generally, it
looks quite unlikely that there exists a simple checkable criterion which works for all
fractal sets in this question. In this paper, we will give some checkable conditions for any self-similar sets with positive similarity ratios.
We now introduce some related results. The first one, to the best of our knowledge, is due to
Steinhaus \cite{HS} who  proved in 1917 the following interesting  results:
$$C+C=\{x+y:x,y\in C\}=[0,2], C-C=\{x-y:x,y\in C\}=[-1,1],$$ where $C$ is the middle-third Cantor set.
In 2019, Athreya, Reznick and Tyson \cite{Tyson} proved that
\[
C\div C=\left\{\dfrac{x}{y}:x,y\in C, y\neq0\right\}=\bigcup_{n=-\infty}^{\infty}\left[ 3^{-n}\dfrac{2}{3},3^{-n}\dfrac
{3}{2}\right] \cup\{0\}.
\]
In \cite{Gu}, Gu, Jiang, Xi and Zhao discussed  the topological structure of $$C\cdot C=\{xy:x,y\in C\}.$$
They proved  the exact Lebesgue measure of $C\cdot C$ is about $0.80955.$   We give some remarks on the above results. The main idea of \cite{Tyson} is only  effective for homogeneous self-similar sets, i.e. all the similarity ratios coincide. For a general self-similar set or some general Cantor set, we may not utilize their idea directly.

We give the definition of thickness \cite{Astels,Falconer}.
 Recall that every Cantor set $K$ on the real line can be constructed by starting with a closed interval $I=I_1$ (the convex hull of $K$), and successively removing disjoint open complementary
intervals. Clearly
there are  countably many disjoint open complementary intervals $(O_n)_n$, which
we may assume, are ordered so that their lengths $|O_n|$ are non-increasing. If several
intervals have the same length, we order them arbitrarily. The two unbounded path-connected
components of $R\setminus K$ are not included. For each $n\in \mathbb{N}$ the interval $O_n$ is a
subset of some closed path-connected component $I_n$ of $I\setminus (O_1\cup O_2\cup \cdots \cup O_{n-1})$.We say
that such a $O_n$ is removed from $I_n$. We call all the removed open intervals $(O_n)_n$ gaps of $K$.
Each $O_n$ is removed from
a closed interval $I_n$, leaving behind two closed intervals $L_n$ and $R_n$, the left and right
of $I_n\setminus O_n$.
 We call $L_n$ and $R_n$ bridge of $K$.
 The thickness of $K$ is defined by
 $$\tau(K)=\inf_{n\in \mathbb{N}}\dfrac{\min\{|L_n|, |R_n|\}}{|O_n|},$$
 where $|\cdot|$ stands for the length of the convex hull.
 We assume that the sequence of complementary intervals $(O_n)_n$ is always infinite, and that a single point is not contained in the Cantor set.

There is another  method which can generate compact sets.
Let $[A,B]$ be a closed interval. In the first step, we remove $n_1-1$ open intervals from $[A,B]$, and obtain $n_1$ closed intervals, where $n_1\in \mathbb{N}_{\geq 2}.$
Let $F_1$ be the union of  all these $n_1$ intervals. In the second stage,  for each remaining closed interval $I_{i}, 1\leq i\leq n_1$, we remove $(n_{2,i}-1), 1\leq i\leq n_1, $ open intervals and obtain  $n_{2,i}$ closed intervals, where $n_{2,i}\in \mathbb{N}_{\geq 2}.$ We denote by $ F_2$  the union of  all the remaining closed intervals in the second stage.
Generally, if we have obtained $F_k$, then for each closed interval from $F_k$, we remove finite open intervals and get finitely many remaining closed intervals. Let $F_{k+1}$ be the union of all these new remaining closed intervals in the $(k+1)$-th step.
We define a set
$$K=\bigcap_{k=1}^{\infty}F_k,$$
and call it a Cantor set.
In some reference, we call such set a Moran set. Throughout the paper, we call it Cantor set.
We call all the deleted open intervals the gaps of $K$, and  all the remaining closed intervals in each $F_k, k\geq1$ the bridges of $K$.
We can also define thickness of $K$ with similar definition.

For a  bridge from some $F_k, k\geq 1$, denoted by $I$, we denote $\hat{I}$ by the unique bridge in $F_{k-1}$ that contains $I$.  We say that $\hat{I}$ is the father of $I$ or $I$ is an offspring of $\hat{I}$.  Similarly, given a gap of in the $k$-th level, $ k\geq 1$, i.e.  $G$, we say $\hat{G}$ is the father of $G$ or $G$ is an offspring of $\hat{G}$, where $\hat{G}$ is the unique bridge in $F_{k-1}$ that contains $G$.
Here we adopt the convention that $F_0=[A,B].$
We write $r_{I}$ for the ratio of $|I|$ to $|K|$.
Given $d\in \mathbb{N}_{\geq 2}$.
Let $\{K_i\}_{i=1}^{d}$ be $d$ Cantor sets.
Define  $$ s_{min}:=\min_{1\leq i\leq d}\inf\{|I|/|\hat{I}|:I \mbox{ is a bridge from  }K_i\}. $$
The number $s_{min}$ can be considered as the minimal ``similarity ratio" of $\{K_i\}_{i=1}^{d}$ for all the steps.
Let $G$ be a gap of some $K_i, 1\leq i\leq d$.  We suppose further that $G$  is not  one of the  deleted gaps in the first step when we construct $K_i$, i.e.  $G\subset F_1^{(i)}$, where $F_1^{(i)}$ denotes the union of all the bridges  of $K_i$ in the first step.

We define the following numbers
 \begin{eqnarray*}
\kappa &=&\max_{1\leq i\leq d}\sup\left\{\dfrac{|G|}{|\hat{G}|}:G \mbox{ is a gap of } K_i,  \mbox{ and  }G\subset F_1^{(i)} \right\}\\
\kappa^{+} &=&\max_{1\leq i\leq d}\sup\left\{\dfrac{|G|}{|\hat{G}|}:G \mbox{ is a gap of } K_i \right\}.
\end{eqnarray*}
Clearly, $\kappa^{+}\geq \kappa.$
Let $$f:\mathbb{R}^d\to \mathbb{R},d\geq 2$$ be a  $C^1$ function, and $D\subset \mathbb{R}^d$ be some compact set.
Define $$L_i=\max_{x\in D}\left|\dfrac{\partial f}{\partial x_i}\right|, S_i=\min_{x\in D}\left|\dfrac{\partial f}{\partial x_i}\right|,$$where $x=(x_1,\cdots,x_d).$
In what follows, we always assume that $$|\partial_{x_i} f|=|f_i|=\left|\dfrac{\partial f}{\partial x_i}\right|>0$$ for any $1\leq i\leq d$.
Define $$r_i=\dfrac{L_i}{S_i}.$$

We  use the partial derivatives to replace the thicknesses. The  following theorem can be viewed as a nonlinear thickness theorem without using  thicknesses.
\begin{theorem}\label{Cantor}
Let $\{K_i\}_{i=1}^{d}\subset \mathbb{R}$   be Cantor sets.
Let the convex hull of  $K_i$ be  $[A_i,B_i]$, $i=1,\cdots,d$.  If for any $ i\in\{1,2,\cdots,d\}$ and any $(x_1,x_2,\cdots, x_d)\in [A_1,B_1]\times\cdots \times [A_d, B_d]$, we have
\begin{equation*}
   \left\lbrace\begin{array}{cc}
\partial_{x_i}f>0\\
s_{min}\left(\sum_{j\neq i}(B_j-A_j)\partial_{x_j}f \right)-\kappa(B_{i}-A_{i}) \partial_{x_i}f\geq 0,
                \end{array}\right.
\end{equation*}
then
$$f(K_1,K_2,\cdots,K_d)=f(F_1^{(1)},F_1^{(2)},\cdots,F_1^{(d)}),$$ i.e.  $f(K_1,K_2,\cdots,K_d)$ is a  union of finitely many closed intervals.
\end{theorem}
When we consider the function $f(x_1,x_2,\cdots,x_d)=\sum_{i=1}^{d}x_i$, we have the following.
\begin{corollary}\label{Cantor-Newhouse}
Let $\{K_i\}_{i=1}^{d}\subset \mathbb{R}$   be Cantor sets.
Let the convex hull of  $K_i$ be  $[A_i,B_i]$, $i=1,\cdots,d$.  If for any $ i\in\{1,2,\cdots,d\}$,
$$
s_{min}\left(\sum_{j\neq i}(B_j-A_j) \right)-\kappa(B_{i}-A_{i}) \geq 0,
              $$
then
$$K_1+K_2+\cdots+K_d=\left\{\sum_{i=1}^{d}x_i:x_i\in K_i,1\leq i\leq d\right\}$$  is a  union of finitely many closed intervals.
\end{corollary}
Comparing with Astels' result, i.e. Theorem \ref{Astels},  in Corollary \ref{Cantor-Newhouse}, the numbers  $s_{min}$ and $\kappa$ are crucial to the structure of $K_1+K_2+\cdots+K_d$. Note that
if $s_{min}\geq \kappa^{+}$, then $$\min_{1\leq i\leq d}\tau(K_i)\geq 1.$$

Applying Theorem \ref{Cantor} to   two Cantor sets, we have the following result.
\begin{corollary}\label{Cantorsum}
Let $\{K_i\}_{i=1}^{2}\subset \mathbb{R}$   be Cantor sets, and the convex hull of  $K_i$ be  $[A_i,B_i]$, $i=1,2$.  If for any $(x,y)\in [A_1,B_1]\times [A_2,B_2]$,
\begin{equation*}
   \left\lbrace\begin{array}{cc}
   \partial_x f>0,\partial_y f>0\\
s_{min}(B_2-A_2)\partial_y f -\kappa(B_{1}-A_{1})\partial_x f \geq 0\\
s_{min}(B_1-A_1)\partial_x f -\kappa(B_{2}-A_{2})\partial_y f \geq 0\\
 \partial_y f(B_2-A_2) \geq \partial_x f\max\{|G_{j}^{(1)}|\}\\
\partial_x f\min\{|B_j^{(1)}|\}\geq\partial_y f\max\{|G_{j}^{(2)}|\},
                \end{array}\right.
\end{equation*}
where $\{G_{j}^{(i)}\}$ denote the gaps of $K_i$ in the first step, and $\{B_j^{(1)}\}$ are the bridges of $K_1$ in the first step,
then
$$f(K_1,K_2)=\{f(x,y):x\in K_1,y\in K_2\}$$  is a  closed interval.
\end{corollary}
\begin{remark}
Comparing with Astels' result, i.e. Theorem \ref{Astels},
Theorem \ref{Cantor} and Corollary \ref{Cantorsum} do not need the linked condition. Without the linked condition, Astels cannot obtain the complete topological structure of  $K_1+K_2+\cdots+K_d$.
The last two conditions in the bracket of Corollary \ref{Cantorsum} guarantee  that $f(F_1^{(1)},F_1^{(2)})$ is a closed interval.
If $\partial_{x_i}f<0$ for some $i$, then we may prove similar results. The proof is almost the same.
If the conditions on the first-order  partial derivatives are not satisfied, then we may consider the higher-order partial derivatives. This is allowed by the Taylor's theorem. Naturally, we need to assume $f\in C^k$ for some integer $k\geq 2.$ The key is to prove the  \textbf{Claim} in the proof of Theorem \ref{Cantor} using the higher-order  partial derivatives.
Our idea is still valid for many other fractal sets, for instance, random self-similar sets, conformal sets and so forth. We do not introduce similar results.
\end{remark}
Theorem \ref{Cantor} provides a method which can judge whether the intersection of  two scaled Cantor sets is empty. 
\begin{theorem}\label{intersection}
Let $\{K_i\}_{i=1}^{2}\subset \mathbb{R}$   be Cantor sets, and the convex hull of  $K_i$ be  $[A_i,B_i]$,  where $A_i>0$, $i=1,2$.  If
\begin{equation*}
   \left\lbrace\begin{array}{cc}
A_1 s_{min}(B_2-A_2) -B_2\kappa(B_{1}-A_{1}) \geq 0\\
A_2 s_{min}(B_1-A_1) -B_1\kappa(B_{2}-A_{2}) \geq 0,\\
                \end{array}\right.
\end{equation*}
then for any real number $a,b\in \mathbb{R}\setminus\{0\}$,
$$aK_1\cap b K_2\neq \emptyset$$
if and only if
$$\dfrac{b}{a}\in \dfrac{F_1^{(1)}}{F_1^{(2)}}=\left\{\dfrac{x}{y}:x\in F_1^{(1)}, y\in F_1^{(2)}\right\}.$$
\end{theorem}
The proof of this result is due to the following fact:
$$aK_1\cap b K_2\neq \emptyset$$ if and only if
$$\dfrac{b}{a}\in \dfrac{K_1}{K_2}.$$ So, if we prove under the conditions in Theorem \ref{intersection} that $\dfrac{K_1}{K_2}$ is a finite union of closed intervals, then we have the desired result. 
Comparing with the Newhouse's gap lemma, Theorem \ref{intersection} can judge  whether  the intersection of two scaled Cantor sets is empty.

Using similar ideas as Theorem \ref{Cantor}, we  partially  describe the topological structure of  the continuous image of any self-similar sets with positive similarity ratios.
\begin{theorem}\label{arithmetic}
Let $\{K_i\}_{i=1}^{d}$   be  self-similar sets with IFS's $$\{\phi_{k,i}(x)=r_{k,i}x+a_{k,i}\}_{k=1}^{\ell_{i}},$$ where $r_{k,i}\in(0,1), a_{k,i}\in \mathbb{R},\ell_i\in \mathbb{N}_{\geq 2}.$  Set $$ s_{min}:=\min\{r_{k,i}:1\leq i\leq d,1\leq k\leq \ell_i\}.$$  Suppose that the convex hull of  $K_i$ is  $[A_i,B_i]$, $i=1,\cdots,d$.  If for any $(x_1,x_2,\cdots,x_d)\in [A_1,B_1]\times \cdots \times [A_d,B_d]$,  and  $1\leq i\leq d,1\leq k\leq \ell_i-1,$
\begin{equation*}
   \left\lbrace\begin{array}{cc}
   \partial_{x_i}f>0\\
s_{min}\left(\sum_{j\neq i}(B_j-A_j)\partial_{x_j}f \right)+(\phi_{k,i}(B_i)-\phi_{k+1,i}(A_i))\partial_{x_i}f \geq 0,
                \end{array}\right.
\end{equation*}
then
$$f(K_1,K_2,\cdots,K_d)=\{f(x_1,x_2,\cdots,x_d):x_i\in K_i,1\leq i\leq d\}$$  is a  union of finitely many closed intervals.
\end{theorem}
\begin{remark}
We emphasize that no separation condition is imposed on the underlying IFS's. Moreover, the similarity ratios are different.
We  compare our result with Athreya,  Reznick, and  Tyson's work. Their idea is  effective for homogeneous self-similar sets, i.e. all the contractive ratios  coincide. We, however, can handle inhomogeneous cases.

\end{remark}
Theorem \ref{arithmetic} implies many results about the arithmetic on self-similar sets. We only give two results.
\begin{corollary}\label{arithmetictwo}
Let $K_1$ and $K_2$ be two self-similar sets with the IFS's
$$\{\phi_i(x)=r_ix+a_i\}_{i=1}^{p}, \{\psi_j(x)=\rho_jx+b_j\}_{j=1}^{q}, \mbox{respectively},r_i,\rho_j\in(0,1),$$where $$ a_i,b_j\in \mathbb{R},p,q\in \mathbb{N}_{\geq 2}.$$  Suppose that the convex hull of  $K_i$ is  $[A_i,B_i]$, $i=1,2$.  If for any $(x_1,x_2)\in [A_1,B_1]\times  [A_2,B_2]$,  and  $1\leq i\leq p-1, 1\leq j\leq q-1$
\begin{equation*}
   \left\lbrace\begin{array}{cc}
   \partial_{x}f>0,\partial_{y}f>0\\
s_{min}\left((B_2-A_2)\partial_{y}f \right)+(\phi_{i}(B_1)-\phi_{i+1}(A_1))\partial_{x}f \geq 0\\
s_{min}\left((B_1-A_1)\partial_{x}f \right)+(\psi_{j}(B_2)-\psi_{j+1}(A_2))\partial_{y}f \geq 0\\
                \end{array}\right.
\end{equation*}
then
$$f(K_1,K_2)=\{f(x,y):x\in K_1,y\in K_2\}$$  is a    closed interval.

\noindent Moreover, if
 there exists some  $(x_0,y_0)\in K_1\times  K_2$  such that for any

 \noindent $1\leq i\leq p-1, 1\leq j\leq q-1$, we have
\begin{equation*}
   \left\lbrace\begin{array}{cc}
   \partial_{x}f|_{(x_0,y_0)}>0,\partial_{y}f|_{(x_0,y_0)}>0\\
s_{min}\left((B_2-A_2)\partial_{y}f|_{(x_0,y_0)} \right)+(\phi_{i}(B_1)-\phi_{i+1}(A_1))\partial_{x}f|_{(x_0,y_0)} \geq 0\\
s_{min}\left((B_1-A_1)\partial_{x}f|_{(x_0,y_0)} \right)+(\psi_{j}(B_2)-\psi_{j+1}(A_2))\partial_{y}f|_{(x_0,y_0)} \geq 0,
                \end{array}\right.
\end{equation*}
then
$f(K_1,K_2)=\{f(x,y):x\in K_1,y\in K_2\}$ contains some interior.
\end{corollary}
By Theorem \ref{arithmetic}, under the conditions in above theorem, $f(K_1,K_2)$ is a union of finitely many closed intervals. However,
under the same conditions, we are allowed to prove
$$f(\cup_{i=1}^{p}\phi_i([A_1, B_1]), \cup_{j=1}^{q}\psi_j([A_2, B_2]))=\cup_{i=1}^{p}\cup_{j=1}^{q}f(\phi_i([A_1, B_1]),\psi_j([A_2, B_2]))$$ is a closed interval. So,
$$f(K_1,K_2)=f(\cup_{i=1}^{p}\phi_i([A_1, B_1]), \cup_{j=1}^{q}\psi_j([A_2, B_2]))$$ is a closed interval.
\begin{corollary}\label{Multiplication}
Let $K_1$ and $K_2$ be two self-similar sets with the IFS's
$$\{\phi_i(x)=r_ix+a_i\}_{i=1}^{p}, \{\psi_j(x)=\rho_jx+b_j\}_{j=1}^{q},r_i,\rho_j\in(0,1), a_i,b_j\in \mathbb{R}.$$
 Suppose that the convex hull of $K_i$ is $[A_i,B_i], A_i\geq 0, i=1,2.$
 If there exists some $(x_0,y_0)\in K_1\times K_2$ such that
 \begin{equation*}
   \left\lbrace\begin{array}{cc}
s_{min}(B_2-A_2)x_0+(\phi_{i}(B_1)-\phi_{i+1}(A_1))y_0>0, 1\leq i\leq p-1\\
s_{min}(B_1-A_1)y_0+(\psi_{j}(B_2)-\psi_{j+1}(A_2))x_0> 0,1\leq j\leq q-1,\\
                \end{array}\right.
\end{equation*}
then
$$K_1
\cdot K_2=\{x\cdot y:x\in K_1, y\in K_2\}$$ contains some interior.
\end{corollary}
\begin{remark}
It deserves to mention the following beautiful formulae\cite{BABA1,Hochman2012}.
Let $K$ be any self-similar set in $\mathbb{R}$. Then
$$\dim_{H}(K\cdot K)=\min\{2\dim_{H}(K), 1\}.$$
Moreover, Hochman and Shmerkin \cite{Hochman2012} proved  the following  general case.

Let $K_1$ and $K_2$ be two self-similar sets with the IFS's
$$\{\phi_i(x)=r_ix+a_i\}_{i=1}^{p}, \{\psi_j(x)=\rho_jx+b_j\}_{j=1}^{q},r_i,\rho_j\in(0,1), a_i,b_j\in \mathbb{R}.$$
If there exist some $r_i$ and $\rho_j$ such that
$$\dfrac{\log r_i}{\log \rho_j}\notin \mathbb{Q},$$
then
$$\dim_{H}(K_1+ K_2)=\dim_{H}(K_1\cdot K_2)=\min\{\dim_{H}(K_1)+\dim_{H}(K_2), 1\}.$$
The above formulae are elegant as they do not need any separation condition. For the critical cases, i.e.
$$\dim_{H}(K_1\cdot K_2)=1,\dim_{H}(K_1+ K_2)=1,$$ it is natural to ask whether
$K_1\cdot K_2$ and $K_1+ K_2$ contain some interiors. Corollaries \ref{arithmetictwo} and \ref{Multiplication} provide some checkable conditions under which $K_1\cdot K_2$ and  $K_1+ K_2$ contain  interiors.
\end{remark}
Now, we consider the linked condition in the Astels' thickness theorem. Without this condition, we still have similar results.
The first result is motivated by the author's previous paper \cite{Jiang2022}. In fact, it is a starting point of analyzing thickness theorems without the linked condition.
\begin{theorem}\label{finitely many}
Let $K_i, i=1,2$ be a Cantor set with convex hull $[A_i, B_i]$. If for any $(x,y)\in [A_1, B_1]\times [A_2, B_2]$,
$$\dfrac{1}{\tau(K_1)}\leq \left|\dfrac{\partial_x f}{\partial_y f}\right|\leq \tau(K_2),$$
then
$f(K_1, K_2)$ is a finite union of closed intervals. Moreover, the number of closed intervals is controlled by
$h_1+v_1+1,$ where
$$h_1=\sharp\left\{O_x:\dfrac{B_2-A_2}{|O_x|}\leq \max_{(x,y)\in [A_1, B_1]\times [A_2, B_2]}\left|\dfrac{\partial_x f}{\partial_y f}\right|\right\}$$
and
$$v_1=\sharp\left\{O_y:\dfrac{|O_y|}{B_1-A_1}\geq \min_{(x,y)\in [A_1, B_1]\times [A_2, B_2]}\left|\dfrac{\partial_x f}{\partial_y f}\right|\right\},$$
 $O_x$ and $O_y$ denote the gaps of $K_1$ and $K_2$, respectively.
\end{theorem}
It is natural to consider a function with multiple variables. The next result is motivated by Moshchevitin\cite{Moshchevitin2000}, who gave  an elegant thickness theorem. However, Moshchevitin does not give a detailed  proof of his thickness theorem.  The linked condition (Moshchevitin called the initial condition) is very complicated.   We combine his  idea with our method, and prove a thickness theorem without linked condition. After removing the linked condition, we find that the proof is similar to that of Theorem \ref{finitely many}.
\begin{theorem}\label{Main}
Let $K_i$ be a Cantor set with convex hull $[A_i,B_i]$, $1\leq i\leq d$. Define a compact set $D=[A_1,B_1]\times \cdots \times [A_d,B_d]$.
If one of the following conditions is satisfied, i.e.
\begin{itemize}
\item [(1)] $$L_i\leq \sum_{j\neq i}S_j\tau(K_j), 1\leq i\leq d;$$
\item [(2)]$$\sum_{i=1}^{d}\dfrac{\tau(K_i)}{\tau(K_i)+r_i}\geq 1,$$
\end{itemize}
then $$f(K_1,K_2,\cdots, K_d)=\{f(x_1,x_2,\cdots, x_d):x_i\in K_i,1\leq i\leq d\}$$ is a union of finitely many closed intervals.
\end{theorem}

The following result generalizes  the first  result of Theorem \ref{Astels} proved by   Astels \cite{Astels}.
\begin{corollary}\label{Astels extension}
Let $K_i$ be a compact set with convex hull $[A_i,B_i]$, $1\leq i\leq d$.
If
$$\sum_{i=1}^{d}\dfrac{\tau(K_i)}{\tau(K_i)+1}\geq 1,$$
then $$K_1+K_2+\cdots+ K_d=\left\{\sum_{i=1}^{d}x_i:x_i\in K_i\right\}$$ is a  union of finitely many closed intervals.
\end{corollary}
\begin{remark}
Astels proved,  under the  same condition in theorem,  that $$K_1+K_2+\cdots+ K_d$$ contains some interiors. We, however, strengthen Astels' result  and obtain the exact structure of $K_1+K_2+\cdots+ K_d$. Generally, without the linked condition, we can only obtain that $K_1+K_2+\cdots+ K_d$ is a finite union of closed intervals. We give one example in Section 3.
Corollary \ref{Astels extension}  implies the following result which generalizes the Newhouse thickness theorem.
Let $\{K_i\}_{i=1}^{2}$ be two Cantor sets, if $\tau(K_1)\tau(K_2)\geq 1$, which is equivalent to $$\dfrac{\tau(K_1)}{1+\tau(K_1)}+\dfrac{\tau(K_2)}{1+\tau(K_2)}\geq1,$$ then
$$K_1+K_2=\{x+y:x\in K_1, y\in K_2\}$$ is a union of finitely many closed intervals.
The above result can be proved by the Newhouse's gap lemma \cite{Takahashi}.
\end{remark}
Theorem \ref{Main} or Theorem \ref{Cantor} yields  the following result.
\begin{corollary}\label{multiplication}
Let $\{K_i\}_{i=1}^{\infty}$ be   Cantor sets, where $K_i$ and $ K_j$ have different convex hulls, $i\neq j.$ If $$\inf_{i\geq 1}\tau(K_i)>0, \min_{1\leq i\leq d} \min(K_i)>0,$$
then there exist $k$ different Cantor sets $\{K_{i_j}\}_{j=1}^{k}\subset \{K_i\}$ such that
$$\Pi_{j=1}^k K_{i_j}=\{\Pi_{j=1}^k x_{i_j}:x_{i_j}\in K_{i_j}\}$$ is a finitely union of closed intervals.
\end{corollary}
Corollary \ref{multiplication} cannot be reproved  directly by Astels \cite{Astels} and Takahashi\cite{Takahashi} as Astels' result only states that the  multiple sum of Cantor sets  comprises interior, provided  that the condition in Corollary \ref{Astels extension} is satisfied.

This paper is arranged as follows. In Section 2, we prove the main results. In Section 3, we give some applications and examples. Finally, we give some remarks and pose some problems.
\section{Proofs of Main theorems}
\begin{proof}[\textbf{Proof of Theorem \ref{Cantor}}]
Let $E_{0}^{(i)}=[A_i, B_i], 1\leq i\leq d$. Suppose we have defined $(E_{k}^{(1)},\cdots ,E_{k}^{(d)} )$ for some $k\geq 0$.
Take a bridge   $I^{\prime}$ from $ E_{k}^{(i)}$ for some $i\in\{1,2,\cdots,d\}$ such that $$\dfrac{|I^{\prime}|}{|\hat{I^{\prime}}|}=\max_{1\leq i\leq d}\max\left\{\dfrac{|I|}{|\hat{I}|}:I \mbox{ is a bridge  taken from }E_{k}^{(i)}\right\}.$$
If $I^{\prime}$ is from some $ E_{0}^{(i)}, 1\leq  i\leq d$, i.e. $I^{\prime}=E_{0}^{(i)}=[A_i, B_i],$ then we let
$\hat{I^{\prime}}=I^{\prime}$. Therefore, $$\dfrac{|I^{\prime}|}{|\hat{I^{\prime}}|}=1.$$

Roughly speaking, comparing with the associated  fathers, the bridge $I^{\prime}$  has  the largest ``similarity ratio".

Then we let $$E_{k+1}^{(i)}=(E_{k}^{(i)}\setminus I^{\prime})\cup \{\cup_{j=1}^{n^{\prime}}I_{j}^{\prime}:n^{\prime}\in \mathbb{N} \}, E_{k+1}^{(j)}=E_{k}^{(j)}, j\neq i,$$
where each $I_{j}^{\prime}$ is one of the offspring of $I^{\prime}$.
By our algorithm, for  $k=d$, each
$E_{k}^{(i)}, 1\leq i\leq d$ is exactly the union of all the bridges of $K_i$ in the first step.

With the construction of $(E_{k}^{(1)},\cdots ,E_{k}^{(d)} )$, we  may prove by induction that for each $k\in \mathbb{N}$,
\begin{equation}\label{keq}
\alpha_{k}\geq \beta_{k},
\end{equation}
where
 \begin{eqnarray*}
 \alpha_k&=&\min_{1\leq i\leq d}\min\left\{\dfrac{|I|}{B_i-A_i}:I\mbox{ is a birdge of } E_{k}^{(i)}\right\},\\
\beta_k&=&s_{\min}\left(\max_{1\leq i\leq d}\max\left\{\dfrac{|I|}{B_i-A_i}:I\mbox{ is a birdge of } E_{k}^{(i)}\right\}\right).
\end{eqnarray*}
Moreover, for  each $i\in \{1,2,\cdots,d\}$, the number of closed intervals in  $E^{(i)}_{k}$ tends to $+\infty$ as $k\to +\infty,$ and the length of each closed interval in  $E^{(i)}_{k}$ goes to $0$ as $k\to +\infty.$

 \noindent   Note that $$K_i=\bigcap_{k=d}^{\infty}E_{k}^{(i)},1\leq i\leq d.$$
Since $E_{k}^{(i)}\supset E_{k+1}^{(i)}$ for any $k\geq 1$ and $f$ is a continuous function, it follows that
$$f(K_1,\cdots, K_d)=\bigcap_{k=d}^{\infty}f(E_{k}^{(1)},\cdots,E_{k}^{(d)}).$$
So, to prove $f(K_1,\cdots, K_d)$ is a union of finitely many  closed intervals, it suffices to prove
\begin{equation}\label{eq1}
f(E_{k}^{(1)},\cdots,E_{k}^{(d)})=f(E_{k+1}^{(1)},\cdots,E_{k+1}^{(d)})
\end{equation} for any $k\geq d$.

By the definition of $$E_{k}^{(i)}, 1\leq i\leq d,$$ for any $k\geq d$  there exists a unique $i_0$ ($1\leq i_0\leq d$) such that
$$E_{k+1}^{(i_0)}\subsetneq E_{k}^{(i_0)}, E_{k+1}^{(j)}=E_{k}^{(j)},j\neq i_0,$$
and
$$E_{k+1}^{(i_0)}=(E_{k}^{(i_0)}\setminus I_{i_0})\cup \{\cup_{j=1}^{n^{\prime}}I_{{i_0}j}:n^{\prime}\in  \mathbb{N}\},$$
where $I_{i_0}$ is a bridge  taken from $E_{k}^{(i_0)}$, and $\cup_{j=1}^{n^{\prime}}I_{{i_0}j}$ is the remaining bridges of $I_{i_0}$ in the next step.
So, it suffices to prove
\begin{equation}\label{eq21}
f(E_{k}^{(1)},\cdots,E_{k}^{(i_0-1)},E_{k+1}^{(i_0)},E_{k,}^{(i_{0}+1)}\cdots, E_{k}^{(d)})= f(E_{k}^{(1)},\cdots,E_{k}^{(i_0)},\cdots, E_{k}^{(d)}).
\end{equation}
 Notice that each $E_{k}^{(i)},1\leq i\leq d$, is a  union of finitely many closed intervals, and we have the following property: for any closed interval  $$A_{i_j}^{(j)}, 1\leq j\leq d, 1\leq i_j\leq k_j, k_j\in \mathbb{N}^{+} ,$$
 $$f(\cup_{i_1=1}^{k_1}A_{i_1}^{(1)}, \cup_{i_2=1}^{k_2}A_{i_2}^{(2)}, \cdots,\cup_{i_d=1}^{k_d}A_{i_d}^{(d)})=\cup{i_1}\cdots\cup_{i_d}f(A_{i_1}^{(1)}, A_{i_2}^{(2)}, \cdots, A_{i_d}^{(d)}).$$
  Therefore, equation (\ref{eq21}) can be proved via the following claim.

\noindent\textbf{Claim:}

 Let $I_j=[a_j,b_j]$ be one of  the closed intervals taken from $ E_{k}^{(j)}, 1\leq j\leq d$.
 In particular,
 $$I_{i_0}=[a_{i_0},b_{i_0}], I_{i_0k}=[a_{i_0k},b_{i_0k}], 1\leq k\leq n^{\prime} \mbox{ for some }n^{\prime}.$$
 We suppose without loss generality that $a_{i_0k}<a_{i_0(k+1)},1\leq k\leq n^{\prime}-1.$
Then
$$f(P)-f(Q)\geq 0$$ for any $1\leq k\leq n^{\prime}-1,$where
 \begin{eqnarray*}
 P&=&(b_1,\cdots, b_{i_0-1}, b_{i_0k},b_{i_0+1} \cdots,b_{d})\\
Q&=&(a_1,\cdots, a_{i_0-1}, a_{i_0k+1},a_{i_0+1} \cdots,a_{d}).
\end{eqnarray*}
By Taylor's theorem, there exists some $(\zeta_1, \cdots, \zeta_d)\in \mathbb{R}^d$
such that
 \begin{eqnarray*}
f(P)-f(Q)&=&\sum_{j\neq i_0}\partial_{x_j}f(\zeta_1, \cdots, \zeta_d) (b_j-a_j) \\
&+&\partial_{x_{i_0}}f(\zeta_1, \cdots, \zeta_d) (b_{i_0k}-a_{i_0k+1}).
\end{eqnarray*}
By (\ref{keq}),  we have $r_{I_j}\geq s_{min}r_{I_{i_0}},j\neq i_0,1\leq j\leq d$.
Note that $$b_j-a_j=|I_j|:=r_{I_j}(B_j-A_j)\geq s_{min}r_{I_{i_0}}(B_j-A_j),$$ and  $-b_{i_0k}+a_{i_0k+1}$ is the length of some gap of $I_{i_0}$ in the next level.
By virtue of  the definition of $\kappa$, we have
$$\dfrac{(-b_{i_0k}+a_{i_0k+1})}{|I_{i_0}|}\leq \kappa.$$
Therefore,
$$(b_{i_0k}-a_{i_0k+1})\geq -|I_{i_0}|\kappa=-r_{I_{i_0}}(B_{i_0}-A_{i_0})\kappa.$$
Hence,
 \begin{eqnarray*}
f(P)-f(Q)&\geq&\sum_{j\neq i_0}\partial_{x_j}f(\zeta_1, \cdots, \zeta_d)s_{min} r_{I_{i_0}}(B_j-A_j) \\
&-&\partial_{x_{i_0}}f(\zeta_1, \cdots, \zeta_d) r_{I_{i_0}}(B_{i_0}-A_{i_0})\kappa\\
&\geq& 0.
\end{eqnarray*}
The last inequality is due to the condition in theorem. We  finish the proof.
\end{proof}

\begin{proof}[\textbf{Proof of Theorem \ref{arithmetic}}]
The proof of Theorem \ref{arithmetic} is similar to that of Theorem \ref{Cantor}.
Let $\Sigma_{*,i},i=1,2,\cdots, d$   denote the collection of finite words over alphabet $\{1,2,\cdots,\ell_{i}\}$.
We suppose that each $\Sigma_{*,i},i=1,2,\cdots, d$ also contains the empty word $\varepsilon$. Let $\phi_{\varepsilon}=id$ be the identity map of $\mathbb{R}$. For a word $I\in \Sigma_{*,i}, i\in\{1,2,\cdots,d\}$, we write $|I|$ for the length of $I$. For simplicity, we let $r_I$ be the similarity ratio of the similitude $\phi_I$.

Let $s_{min}=\min\{r_{k,i}:1\leq i\leq d,1\leq k\leq \ell_i\}$.
Now, we  recursively construct a sequence $\{(\Omega_{k,1},\Omega_{k,2},\cdots,\Omega_{k,d})\}_{k\geq 1}$ of a $d$-tuple of subsets of $\Sigma_{*,1}\times \cdots \times \Sigma_{*,d}$. First, we let $$\Omega_{1,1}=\Omega_{1,2}=\cdots=\Omega_{1,d}=\{\varepsilon\}.$$
Suppose we have defined the tuple
$(\Omega_{k,1},\Omega_{k,2},\cdots,\Omega_{k,d})$ for some $k$.  We take one word $I^{\prime}$ from $\cup_{i=1}^{d}\Omega_{k,i}$ such that its corresponding  similarity ratio, denoted by $L_{I^{\prime}}$,
 is maximal among the associated similitudes in $$\bigcup_{1\leq i\leq d}\bigcup_{I\in\Omega_{k,i}}\phi_{I} .$$
Namely,
$$L_{I^{\prime}}=\max\{r_{I}:I\in\Omega_{k,i},1\leq i\leq d\}.$$
Suppose $I^{\prime}\in \Omega_{k,i}$ for some $i$, $1\leq i\leq d$,  then we let
$$\Omega_{k+1,i}=(\Omega_{k,i}\setminus I^{\prime})\cup \{I^{\prime}j: j=1,2,\cdots, \ell_i\}$$
and
$$\Omega_{k+1,j}=\Omega_{k,j},  j\neq i.$$
By the above construction, for each $k\geq 1$, we  may prove by induction that
$$\min\{r_{I}:I\in\Omega_{k,i},1\leq i\leq d\}\geq  s_{min}\max\{r_{I}:I\in\Omega_{k,i},1\leq i\leq d\}.$$
Moreover,
$$\inf\{|I|:I\in\Omega_{k,i}\}\to \infty \mbox{ as }k\to \infty.$$

 \noindent   Note that $$K_i=\bigcap_{k=1}^{\infty}E_{k,i},1\leq i\leq d,$$ where
$$E_{k,i}=\cup_{I\in \Omega_{k,i}}\phi_{I}([A_i,B_i]).$$
Since $E_{k,i}\supset E_{k+1,i}$ for any $k\geq 1$ and $f$ is a continuous function,
$$f(K_1,\cdots, K_d)=\bigcap_{k=d}^{\infty}f(E_{k,1},\cdots,E_{k,d}).$$
Hence, to prove $f(K_1,\cdots, K_d)$ is a union of finitely many  closed intervals, it suffices to prove
\begin{equation}\label{eq1}
f(E_{k,1},\cdots,E_{k,d})=f(E_{k+1,1},\cdots,E_{k+1,d})
\end{equation} for any $k\geq d$.

By the definition of $$\Omega_{k,i}, 1\leq i\leq d,1\leq k\leq \ell_i,$$  for any $k\geq 1$  there exists a unique $i_0$ ($1\leq i_0\leq d$) such that
$$\Omega_{k+1,i_0}\subsetneq \Omega_{k,i_0}, \Omega_{k+1,j}=\Omega_{k,j},  j\neq i_0.$$
We suppose
$$\Omega_{k+1,i_0}=(\Omega_{k,i_0}\setminus I^{\prime})\cup \{I^{\prime}j: j=1,2,\cdots, \ell_{i_0}\}.$$
Thus,
$$E_{k+1,i_0}\subsetneq E_{k,i_0}, E_{k+1,j}=E_{k,j},j\neq i_0,$$
and
 $$E_{k+1,i_0}=\cup_{I\in \Omega_{k,i_0}\setminus \{I^{\prime}\}}\phi_{I}([A_{i_0},B_{i_0}])\cup ( \cup_{j=1}^{\ell_{i_0}}\phi_{I^{\prime}j}([A_{i_0},B_{i_0}])).$$
Hence, to prove equation(\ref{eq1}), it suffices to prove
\begin{equation}\label{eq2}
f(E_{k,1},\cdots,E_{k,i_0-1},E_{k+1,i_0},E_{k,i_0+1}\cdots, E_{k,d})= f(E_{k,1},\cdots,E_{k,i_0},\cdots, E_{k,d})
\end{equation}
for any $k\geq d.$
 Notice that each $E_{k,i},1\leq i\leq d$, is a  union of finitely many closed intervals. So,
 equation (\ref{eq2}) can be proved via the following claim.

\noindent\textbf{Claim:}

 Let $\phi_{I_j}([A_j, B_j])$ be one of  the closed intervals taken from $ E_{k,j},I_j\in \Omega_{k,j}, 1\leq j\leq d$.
Then
$$f(P_1)-f(Q_1)\geq 0$$ for any $1\leq k\leq \ell_{i_0}-1,$where
 \begin{eqnarray*}
P_1&=&(\phi_{I_1}(B_1),\cdots, \phi_{I_{i_0-1}}(B_{i_0-1}), \phi_{I_{i_0}k}(B_{i_0}),\phi_{I_{i_{0}+1}}(B_{i_0+1}) \cdots,\phi_{I_d}(B_d)) \\
Q_1&=&(\phi_{I_1}(A_1),\cdots, \phi_{I_{i_0-1}}(A_{i_0-1}), \phi_{I_{i_0}(k+1)}(A_{i_0}),\phi_{I_{i_0+1}}(A_{i_0+1}) \cdots,\phi_{I_d}(A_d)).
\end{eqnarray*}
 In other words, if the  \textbf{Claim} is proved, then we finish the proof of equation (\ref{eq2}).
By Taylor's theorem, there exists some $(\eta_1, \cdots, \eta_d)\in \mathbb{R}^d$
such that
 \begin{eqnarray*}
f(P_1)-f(Q_1)&=&\sum_{k\neq i_0}\partial_{x_k}f(\eta_1, \cdots, \eta_d) r_{I_k}(B_k-A_k) \\
&+&\partial_{x_{i_0}}f(\eta_1, \cdots, \eta_d) r_{I_{i_0}}(\phi_{k,i_0}(B_{i_0})-\phi_{k+1,i_0}(A_{i_0})).
\end{eqnarray*}
Note that for any $I\in \Omega_{k,i}, 1\leq i\leq d $, $$r_{I}\geq s_{min}r_{I_{i_0}}=s_{min}r_{I^{\prime}}, \mbox{ and }\partial_{x_{i}}f>0.$$ Therefore, by the conditions in Theorem \ref{arithmetic}
 \begin{eqnarray*}
f(P_1)-f(Q_1)&\geq&\sum_{k\neq i_0}\partial_{x_k}f(\eta_1, \cdots, \eta_d)s_{min} r_{I_{i_0}}(B_k-A_k) \\
&+&\partial_{x_{i_0}}f(\eta_1, \cdots, \eta_d) r_{I_{i_0}}(\phi_{k,i_0}(B_{i_0})-\phi_{k+1,i_0}(A_{i_0}))\\
&= & r_{I_{i_0}}\left(\sum_{k\neq i}\partial_{x_k}f(\eta_1, \cdots, \eta_d)s_{min}(B_k-A_k)\right) \\
&+&r_{I_{i_0}}\partial_{x_i}f(\eta_1, \cdots, \eta_d) (\phi_{k,i_0}(B_{i_0})-\phi_{k+1,i_0}(A_{i_0}))\\
&\geq& 0.
\end{eqnarray*}
The last inequality is due to the condition in theorem. We  finish the proof.
\end{proof}
We now prove Theorem \ref{Main} for the case $d=2$. This simple case is helpful to understand the  complete proof of  Theorem \ref{Main}.
\begin{theorem}\label{simple case}
Let $K_i, i=1,2$ be a Cantor set with convex hull $[A_i, B_i]$. If for any $(x,y)\in [A_1, B_1]\times [A_2, B_2]$,
we have one of the following conditions, i.e.
\begin{itemize}
\item[(1)]   $$\tau(K_1)\tau(K_2)\geq r_1r_2$$
\item[(2)] $$\dfrac{1}{\tau(K_1)}\leq \left|\dfrac{\partial_x f}{\partial_y f}\right|\leq \tau(K_2),$$
\end{itemize}
then
$f(K_1, K_2)$ is a finitely union of closed intervals.
\end{theorem}
\begin{proof}
We only prove the statement under the first condition. The second is similar.
First, it is easy to prove $$f(K_1, K_2)\subset H:=f([A_1, B_1],  [A_2, B_2]) $$ as $K_i\subset [A_i,B_i], 1\leq i\leq 2.$
 If  $w\in H$ and $w\notin f(K_1,K_2)$, then
the curve  $f(x,y)=w$ does not meet $K_1\times K_2.$
We define
$$\Psi_w=\{(x,y)\in [A_1, B_1]\times  [A_2, B_2]: f(x, y)=w\}.$$
Hence, we can find   countably  many rectangles of the form $$ R=O_x \times  [A_2, B_2]   \mbox{ or } [A_1, B_1]\times O_y$$ such that
 $\Psi_w$ is covered by these rectangles,
 where $O_x$ and  $O_y$  denote  the deleted gaps from  $K_1$ and $K_2$, respectively.
 By the continuity of $f$,  $\Psi_w$ is a compact set.
So we can find finitely many rectangles  $$R_1,R_2,\cdots, R_n$$ such that
$$\Psi_w\subset \cup_{i=1}^{n}R_i.$$
We split  the proof in two cases.

 \textbf{Case 1.} If the finite covering comprises   both of the  vertical and horizontal rectangles, i.e. it simultaneously  consists of the following form,
$$ O_x \times  [A_2, B_2] ,   [A_1, B_1]\times O_y.$$
Let $$\phi_x=L_x+S_x\tau_x,\phi_y=L_y+S_y\tau_y,$$
where
$\tau_x=\tau(K_1), \tau_y=\tau(K_2), D=[A_1, B_1]\times [A_2, B_2]$, and
$$L_x=\max_{(x,y)\in D}\left|\dfrac{\partial f}{\partial x}\right|, S_x=\min_{(x,y)\in  D}\left|\dfrac{\partial f}{\partial x}\right|, r_1=r_x=\dfrac{L_x}{S_x},$$
$$L_y=\max_{(x,y)\in D}\left|\dfrac{\partial f}{\partial y}\right|, S_y=\min_{(x,y)\in  D}\left|\dfrac{\partial f}{\partial y}\right|, r_2=r_y=\dfrac{L_y}{S_y}.$$
Without loss of generality,  we assume that for some $O_y$,
 $$|O_y|\phi_y\leq |O_x|\phi_x$$ holds for any $O_x$.

 By the implicit function theorem and the condition $$\tau(K_1)\tau(K_2)\geq L_iS_i>0,$$ the
 curve $\Psi_w$ determines a local function,
 denoted by $y=g(x)$, on some open interval. By the Heine-Borel covering theorem, i.e. any open cover of a closed interval must have a finite subcover,  the function is indeed global. As such $y=g(x)$ is a function defined on $[A_1,B_1].$
 By our assumption, for any $(x,y)\in [A_1, B_1]\times [A_2, B_2]$,
$|f_x|>0, |f_y|>0 $. Since $f(x,y)\in C^1$, it follows that  $f_x$ and $f_y$ are monotonic with respective to $x$ and $y$, respectively. Therefore, $\dfrac{dy}{dx}=-\dfrac{f_y}{f_x}\neq 0$, which yields that    $y=g(x)$ is monotonic. In the remaining proof, we always suppose $f_x<0, f_y>0$ for any $(x,y)\in [A_1, B_1]\times [A_2, B_2]$, i.e.  $g(x)$ is increasing.

 Since the curve $\Psi_w$ is contained in $\cup_{i=1}^{n}R_i,$ if the curve meets the boundary of some
 $$[A_1,B_1]\times O_y \mbox{ or }O_x\times [A_2,B_2], $$ for instance, we may suppose that
 $\Psi_w$ intersects with $$[A_1,B_1]\times O_y=[A_1,B_1]\times (a_y,b_y)$$ at the point $(x_0,y_0)$, where $y_0=b_y$, then there must exist a vertical rectangle $O_x\times [A_2,B_2]$ such that $(x_0,y_0)\in O_x\times [A_2,B_2]$.
 We may suppose the curve   $\Psi_w$ enters and leaves the rectangle as described in Figure 1. Hence, we can find some $x\in [A_1,B_1]$ and a bridge of $K_1$ with length $\rho_x$ such that
 \begin{equation}\label{keyinequality}
 |O_y| >|g(x+\rho_x)-g(x)|.
 \end{equation}

 \begin{figure}[htpb]
    \centering
    \includegraphics[width=0.4\textwidth]{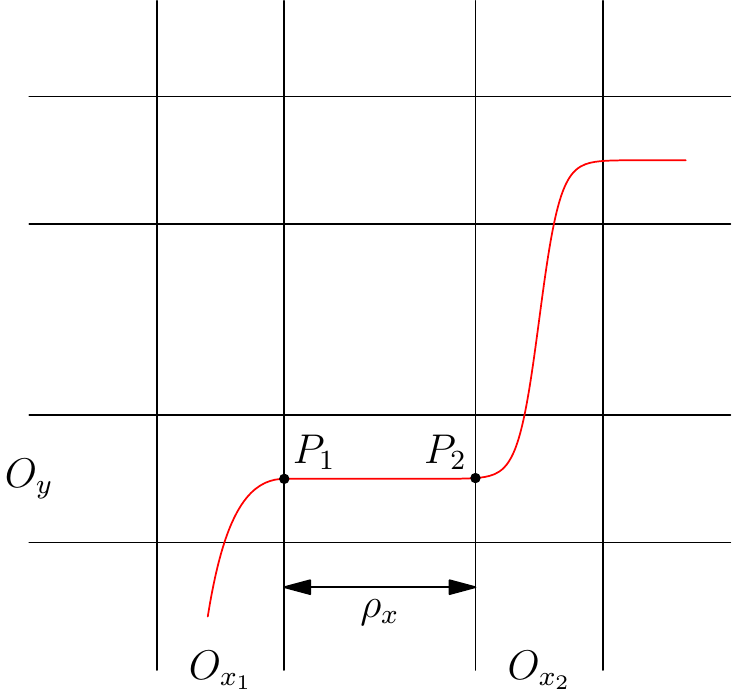}
    \caption{The curve enters and leaves $[A_1,B_1]\times O_y$.}
    \label{fig:tikzpgf}
\end{figure}

 By the mean value theorem, there exists some $\eta\in (x, x+\rho_x)$ such that
 $$|g(x+\rho_x)-g(x)|=|g^{\prime}(\eta)\rho_x|=\left|-\dfrac{f_x}{f_y}\rho_x\right|\geq \left|\dfrac{f_x}{f_y}|O_x|\tau_x\right|, $$
 where $O_x$ is the adjoint (or one of the gaps, i.e. in Figure 1, $O_x=O_{x_1} $ or $O_x=O_{x_2} $ ) gap of  the bridge,  and we use the fact $\rho_x\geq |O_x|\tau_x$ in terms of  the definition of thickness.
 Therefore,
 $$L_y|O_y|\geq |f_y||O_y| >|f_x|O_x|\tau_x|\geq S_x|O_x|\tau_x.$$
 By the definition of $\phi_y$,
 \begin{equation}\label{two}
 \phi_y |O_y|>S_x|O_x|\tau_x+S_y|O_y|\tau_y.
 \end{equation}
 Note that
 $$\dfrac{|O_x|}{\phi_y}\leq \dfrac{|O_y|}{\phi_x}.$$
 This together with (\ref{two}) imply that
 $$|O_y|>\dfrac{S_x|O_y|\tau_x}{\phi_x}+\dfrac{S_y|O_y|\tau_y}{\phi_y},$$
 i.e.
 $$1>\dfrac{\tau_x}{\tau_x+r_x}+\dfrac{\tau_y}{\tau_y+r_y}\Leftrightarrow \tau_x\tau_y<r_xr_y.$$
 We find a contradiction.

 \textbf{Case 2.} If  $\Psi_w$ is contained in a unique $O_x\times [A_2, B_2]$ or $[A_1, B_1]\times O_y$,  in this case, we shall prove that
 the parameter of the curve $\Psi_w$, i.e. $w$,  only takes values from a finite union of open intervals. Therefore, we show $f(K_1, K_2)$ is a finite union of closed intervals.

First, since $f(x,y)$ is $C^{1}$, it follows that the derivative of the curve
$\Psi_w$  is bounded, i.e. there exist some constants $c_1,c_2$ such that
$$0<c_1\leq \left|\dfrac{dy}{dx}\right|=\left|-\dfrac{f_x}{f_y}\right|\leq c_2.$$
Hence, the curve $\Psi_w$ cannot lie in some $[A_1, B_1]\times O_y$ (if the curve lies in some $ O_x\times [A_2, B_2]$, the discussion is similar), where $|O_y|<(B_1-A_1)c_1$. As for this case, by the Lagrange theorem, there exists some $x_0\in [A_1, B_1]$ such that
the derivative of the curve $\Psi_w$ at $x_0$ should be  smaller than
$$\dfrac{|O_y|}{B_1-A_1}<c_1.$$
Hence, for any $w$, if   the curve  $\Psi_w$ is contained in a unique $O_x\times [A_2, B_2]$ or $[A_1, B_1]\times O_y$, then there exist only finitely many such rectangles.
Now, we prove that the parameter of the curve $\Psi_w$, i.e. $w$,  only takes values from a finite union of open intervals.
Recall that we assume
 $y=g(x)$ is increasing.
For two different parameters $w_1$ and $w_2$, we consider $$f(x,y)=w_1 \mbox{ and }f(x,y)=w_2.$$
Let $y=g(x)$ and $y=h(x)$ be two functions of these two curves, respectively.
We note by the implicit function theorem that $$g^{\prime}(x)=h^{\prime}(x)$$ for any $x\in [A_1, B_1]$.
 By the Lagrange theorem, $g(x)=h(x)+c$, where $c$ is a translation.
Hence, if the curve $\Psi_w$ lies in some $$[A_1, B_1]\times O_y=[A_1, B_1]\times (a_y,b_y),$$ then we may translate the curve vertically, and
 $c$ can only takes value  in $(0,g(A_1)-a_y+b_y-g(B_1))$.
 Hence,  the parameter $w$ can be viewed as a function of
 $c$, i.e.
 $$w(c)=f(x,g_0(x)+c),$$ where $g_0(x)=g(x)-c_0$, and $c_0$ is some translation such that $$g_0(A_1)=a_y.$$
 Since $f\in C^1$, it follows from the mean value theorem that $w(c)$ is a continuous function. Recall that  $f_x<0, f_y>0$ for any $(x,y)\in [A_1, B_1]\times [A_2, B_2]$. Thus, by the mean value theorem again,  $w(c)$ is an increasing function. Since $c\in (0,g(A_1)-a_y+b_y-g(B_1))$, the range of  $w(c)$ is also an open interval.
 We finish the proof.

 For the second statement, we can prove in a similar way. First, under the condition
 $$\dfrac{1}{\tau(K_1)}\leq \left|\dfrac{\partial_x f}{\partial_y f}\right|\leq \tau(K_2),$$
 we claim that the \textbf{Case 1} cannot occur. The proof is essentially the same as the first statement.   Second, for the \textbf{Case 2}, the proof is exactly the same. We omit the details.

 Finally, the number of closed intervals in Theorem \ref{finitely many} has already been proved via \textbf{Case 2}.
\end{proof}
\begin{proof}[\textbf{Proof of Theorem \ref{Main}}]
We only prove this result under the  second condition. For the first condition, the idea is the same.
First, it is easy to prove $$f(K_1,\cdots,  K_d)\subset H:=f([A_1, B_1]\times \cdots \times [A_d, B_d]). $$
 If  $w\in H$ and $w\notin f(K_1,\cdots,  K_d)$, then
the hypersurface $f(x_1,\cdots,  z)=w$ does not intersect with $$K_1\times K_2\times\cdots \times K_d.$$
Now  we construct the following set
$$\Psi_w=\{(x_1,x_2,\cdots, x_d)\in [A_1, B_1]\times \cdots \times [A_d, B_d]: f(x_1,\cdots,  x_d)=w\}.$$
Hence, we can find   countably many $d$-dimensional cubes of the form $$ C=[A_1, B_1]\times \cdots\times\Delta_i\times   \cdots \times [A_d, B_d], 1\leq i\leq d,$$ where $\Delta_i$ denotes the deleted open interval from  $K_i$.
Namely, only the $i$-th direction of the cube is an open set, and the rest directions are the convex hull of $K_j,j\neq i.$

By the continuity of $f$,  $\Psi_w$ is a compact set. Therefore,   there exist finitely many  $d$-dimensional cubes, i.e.  $$C_1,C_2,\cdots, C_n$$ such that
$$\Psi_w\subset \cup_{i=1}^{n}C_i.$$
We discuss   in two cases.

 \textbf{Case 1.} If the finite covering consists of all the directions, then we shall find some contradiction.
We write $|\Delta_i|$ for the length of the open interval from the $i$-th direction,  and $\rho_i$ for the length of the appropriate adjoint bridge of $\Delta_i$, where $1\leq i\leq d$.

 By the finiteness of the covering, we can find some $i$, such that
 $$|\Delta_i|\phi_i\leq |\Delta_j|\phi_j, \mbox{ for all } j\neq i,$$
where  $$\phi_i=L_i+S_i\tau(K_i),1\leq i\leq d.$$
 Now, we  prove that
 \begin{equation}\label{key1}
 L_i|\Delta_i|>\sum_{j\neq i}S_j\tau(K_j)|\Delta_j|.
 \end{equation}
 First, by the implicit function theorem and the Heine-Borel covering theorem, we can find a unique global  function $$x_i: [A_1,B_1]\times\cdots \times[A_{i-1}, B_{i-1}]\times [A_{i+1}, B_{i+1}]\times \cdots \times [A_d,B_d]\to \mathbb{R}.$$  More specifically,
 $$x_i=G(x_1,\cdots, x_{i-1}, x_{i+1},\cdots,x_d), 1\leq i\leq d.$$
Note that the hypersurface $$f(x_1,\cdots, x_d)=w$$ is contained in the union of   some  cubes with  the form $$ C_j=[A_1, B_1]\times \cdots\times\Delta_j\times   \cdots \times [A_d, B_d], 1\leq j\leq d.$$
By assumption,  each direction is contained.  For the hypersurface $$f(x_1,\cdots, x_d)=w$$
we suppose it determines a unique function
$$x_i=G(x_1,\cdots, x_{i-1}, x_{i+1},\cdots,x_d),$$ where  for the index $i$, we have
 $$|\Delta_i|\phi_i\leq |\Delta_j|\phi_j, \mbox{ for all } j\neq i.$$
 The hypersurface $$f(x_1,\cdots, x_d)=w$$ partially lies in the cube
 $$\Lambda= [A_1, B_1]\times \cdots\times\Delta_i\times   \cdots \times [A_d, B_d].$$
Thus, we can find some points $X_1, X_2$ from
 $$ [A_1, B_1]\times \cdots\times[A_{i-1}, B_{i-1}]\times  [A_{i+1}, B_{i+1}]  \cdots \times [A_d, B_d],$$ i.e.
 \begin{eqnarray*}
X_1&=&(x_1+\delta_1\rho_1,\cdots,x_{i-1}+\delta_{i-1}\rho_{i-1},x_{i+1}+\delta_{i+1}\rho_{i+1},\cdots,x_{d}+\delta_{d}\rho_{d}) \\
X_2&=&(x_1,\cdots, x_{i-1}, x_{i+1},\cdots, x_d).
\end{eqnarray*}
 such that
 $$|\Delta_i|>|G(X_1)-G(X_2)|=\left|\sum_{j\neq i}-\dfrac{f_{x_j}}{ f_{x_i}}\delta_j \rho_j\right|,$$
 where
 \begin{equation*}
\delta_j=\left\lbrace\begin{array}{cc}
1, f_{x_j}>0\\
-1, f_{x_j}<0.
 \end{array}\right.
\end{equation*}
The reader may refer to inequality (\ref{keyinequality}). The idea is similar.
  Equivalently, we have
$$|f_{x_i}||\Delta_i|>\left|\sum_{j\neq i}f_{x_j}\delta_j \rho_j\right|.$$
By the definition of $\delta_j$,
$f_{x_j}\delta_j>0$. Hence,
$$L_i|\Delta_i|\geq |f_{x_i}||\Delta_i|>\left|\sum_{j\neq i}f_{x_j}\delta_j \rho_j\right|\geq \sum_{j\neq i}S_j \rho_j,$$ i.e.
$$L_i|\Delta_i|>\sum_{j\neq i}S_j\rho_j.$$
  By the definition of Newhouse thickness,  $\rho_j\geq \tau(K_j)|\Delta_j|,j\neq i.$
 Therefore,
 \begin{equation}\label{>>>}
 L_i|\Delta_i|>\sum_{j\neq i}S_j\tau(K_j)|\Delta_j|.
 \end{equation}
 By (\ref{>>>}),
  \begin{equation}\label{>>}
 \phi_i|\Delta_i|=L_i|\Delta_i|+S_i\tau(K_i)|\Delta_i|>\sum_{j=1}^{d}S_j\tau(K_j)|\Delta_j|.
  \end{equation}
 Recall that
 $$|\Delta_i|\phi_i\leq |\Delta_j|\phi_j, \mbox{ for all } j\neq i.$$
Rewriting  inequality (\ref{>>}), we have
 $$1>\sum_{j=1}^{d}\dfrac{S_j\tau(K_j)|\Delta_j|}{\phi_i|\Delta_i|}\geq \sum_{j=1}^{d}\dfrac{S_j\tau(K_j)|\Delta_j|}{\phi_j|\Delta_j|}= \sum_{j=1}^{d}\dfrac{S_j\tau(K_j)}{\phi_j}=\sum_{j=1}^{d}\dfrac{\tau(K_j)}{\tau(K_j)+r_j},$$
which  contradicts the assumption $$\sum_{i=1}^{d}\dfrac{\tau(K_i)}{\tau(K_i)+r_i}\geq 1.$$

Now, we consider \textbf{Case 1} under the first condition:
$$L_i\leq  \sum_{j\neq i}S_j\tau(K_j), 1\leq i\leq d.$$
When we choose $\Delta_i$, we assume that its length is minimal among all the rectangles, i.e.
$$|\Delta_i|\leq |\Delta_j|, i\neq j.$$
Therefore, we still have
 \begin{equation}\label{>}
 L_i|\Delta_i|>\sum_{j\neq i}S_j\tau(K_j)|\Delta_j|.
 \end{equation}
This together with  the minimality of the length of $\Delta_i$ imply that
  $$L_i|\Delta_i|>\sum_{j\neq i}S_j\tau(K_j)|\Delta_j|\geq \sum_{j\neq i}S_j\tau(K_j)|\Delta_i|.$$
  Therefore,
  $$L_i> \sum_{j\neq i}S_j\tau(K_j).$$
We find a contradiction again.

 \textbf{Case 2.}
  If some directions are not included in the finite covering, then we prove the parameter  of the curve $\Psi_w$, i.e. $w$, can only be taken from a finite union of open intervals. We denote these open intervals by $\cup_{j=1}^{k_0}O_j$, where $k_0\in \mathbb{N}^{+}$ is a fixed number.  Therefore,  $$f(K_1,\cdots, K_2)=f([A_1,B_1],\cdots, [A_d,B_d])\setminus (\cup_{j=1}^{k_0}O_j),$$ i.e. it is a union  of finitely many closed intervals.

  We prove the above statement by some cases.

  \textbf{Case 2.1}
  If only one direction is included, without loss of generality, we may suppose
  the covering of $\Psi_w$ is of the following form:
  $$\Delta\times [A_2, B_2]\times \cdots\times [A_d, B_d],$$
  where $\Delta$ is some deleted open interval from $K_1$.
  Note that for any two gaps of $K_1$, i.e. $\Delta_k$ and $\Delta_l$, we have  $\Delta_k\cap \Delta_l=\emptyset, k\neq l.$
  Hence, $$(\Delta_k\times [A_2, B_2]\times \cdots\times [A_d, B_d])\cap (\Delta_l\times [A_2, B_2]\times \cdots\times[A_d, B_d])=\emptyset,k\neq l.$$
  In other words, by the continuity of $f(x_1,\cdots,x_d)$, the hyperplane  $\Psi_w$ can be covered by a unique $$\Gamma=\Delta\times [A_2, B_2]\times \cdots\times [A_d, B_d].$$
  By the implicit function theorem and Heine-Borel covering theorem,
 the hypersurface $\Psi_w$  uniquely  determines the function
  $x_1=U(x_2,x_3,\cdots, x_d)$,  where $U$ is a continuous function. Hence,
  $$|\Delta|> U_{\max}-U_{\min},$$
  where $U_{\max}$ and $U_{\min}$ stand for the maximal and minimal values of $U$.

  Given two parameters $w_1$ and $w_2$, we consider the hypersurfaces
  $$f(x_1,x_2,\cdots, x_d)=w_1,f(x_1,x_2,\cdots, x_d)=w_2.$$
  Suppose they determine the function $U(x_2,x_3,\cdots, x_d)$ and $V(x_2,x_3,\cdots, x_d)$, respectively. Note that
  $$\dfrac{\partial U}{\partial x_{j}}=\dfrac{\partial V}{\partial x_{j}}, j=2,3,\cdots,d.$$
  Let $W=U-V$.  Applying the mean value theorem to $W$,  there exists some constant(translation)
  $c$ such that
  $$U=V+c.$$
  Hence, we translate the hypersurface $x_1=U(x_2,x_3,\cdots, x_d)$ on $\Gamma$, and the parameter $w$ can only be in a finite open interval. The reason why we obtain the open interval is due to the open set $\Delta$, i.e.  the maximal value of the function $x_1=U(x_2,x_3,\cdots, x_d)$ cannot reach the boundary of $\Gamma$.
  Hence, we have proved that if only one direction is included, then
  $f(K_1,K_2,\cdots, K_d)$ is a finitely union of closed intervals.

  \textbf{Case 2.2}

   If   $k$ directions  are included, where $k\in \{2,\cdots,d-1\}$,  without loss of generality, we may suppose
  the covering of $\Psi_w$ is of the following form:
  $$[A_{1}, B_{1}]\times\cdots \times\Delta_{j}\times [A_{j+1}, B_{j+1}]\times\cdots \times [A_{d}, B_{d}],$$
  where $\Delta_j, 1\leq j\leq k,$ is a gap of $K_j$, and  for the last $d-k$ directions, we always use $[A_{k+1}, B_{k+1}]\times\cdots \times[A_d, B_d]$.
  First, we claim that the length of each gap $\Delta_j$  of $K_j, 1\leq j\leq k$ has a  uniform lower bound. The proof is similar to the   \textbf{Case 2.1.}
  By the implicit function theorem and the Heine-Borel covering theorem, the equation
  $$f(x_1,x_2,\cdots, x_d)=w$$ determines a unique continuous function $$x_j=U(x_1,\cdots, x_{j-1}, x_{j+1}, \cdots, x_d)$$ on $[A_1, B_1]\times \cdots \times[A_{j-1}, B_{j-1}]\times [A_{j+1}, B_{j+1}]\times \cdots \times[A_d, B_d], 1\leq j\leq k $.
  Hence, the function $U$ has the maximal and minimal values on its domain (we denote them by $U_{\max}$ and $U_{\min}$, respectively), and  $|\Delta_j|\geq U_{\max}-U_{\min}.$

  Hence, for any $w$, the covering of  the hypersurface $\Psi_w$ should satisfy a basic condition, i.e.  each $|\Delta_j|\geq c_3>0$ for some uniform $c_3$, where  $1\leq j\leq  k$.
  In other words, we have shown the following key result.
  \begin{lemma}\label{finite cover}
There exist finite sets of the following form
$$[A_{1}, B_{1}]\times\cdots \times\Delta_{j}\times [A_{j+1}, B_{j+1}]\times\cdots \times [A_{d}, B_{d}],$$
such that for any $w$,  hypersurface $\Psi_w$   can be uniformly  covered by the above finite covering,
  where $\Delta_j,1\leq j\leq k$ is some gap of $K_j$.
  \end{lemma}

  We note that for any constant $a\in \mathbb{R}$,  the hypersurfaces
  $$f(x_1,x_2,\cdots, x_d)=w$$
  and
  $$f(x_1,x_2,\cdots, x_d)=a$$ determine the same function up to a translation. Namely, they determine  functions $$x_w=U(x_1,x_2,\cdots,x_{j-1}, x_{j+1},\cdots,x_d)$$ and $$x_a=V(x_1,x_2,\cdots,x_{j-1}, x_{j+1},\cdots,x_d),$$ respectively, where $k+1\leq j\leq d$.   These two functions coincide up to a translation $c$, i.e.
  $$x_w=x_a+c.$$
  The proof is similar to the discussion in \textbf{Case 2.1}. We omit the details.
 Lemma \ref{finite cover} and the above discussion imply the following.
  \begin{lemma}\label{translation}
  If we translate the hypersurface $x_w$, then we can find all possible choices of  $w$ such that
 the translated  hypersurface, i.e.
  $$x_w+c,$$  is always
  is contained in the uniform finite covering of  hypersurface
  $$f(x_1,x_2,\cdots, x_d)=w.$$
  \end{lemma}
Without loss of generality, we assume the hypersurface
  $$f(x_1,x_2,\cdots, x_d)=w$$ determines the function
  $$x_w:=x_j=U(x_1,x_2,\cdots,x_{j-1}, x_{j+1},\cdots,x_d),$$ for some $j\in\{k+1,\cdots, d\}$.
By the implicit function theorem, we have
$$\dfrac{\partial x_j}{\partial x_i}=-\dfrac{ f_i}{ f_j}, i\neq j, k+1\leq j\leq d.$$
Since $f(x_1,\cdots, x_d)\in C^1$ and $| f_i|>0, 1\leq i\leq d$, it follows that
$$ f_i>0 \mbox{ or } f_i<0, 1\leq i\leq d$$ on $[A_1,B_1]\times \cdots\times [A_d, B_d].$
We suppose without loss of generality that  $$ f_i<0,1\leq i\leq d.$$
Hence,
$$\dfrac{\partial x_j}{\partial x_i}=-\dfrac{ f_i}{ f_j}<0, i\neq j, k+1\leq j\leq d.$$
So, the maximal and minimal values of $$x_j=U(x_1,x_2,\cdots,x_{j-1}, x_{j+1},\cdots,x_d)$$ on $[A_1,B_1]\times \cdots \times [A_{j-1}, B_{j-1}]\times[A_{j+1}, B_{j+1}] \times \cdots \times [A_d, B_d]$  are $$U(A_1, \cdots, A_{j-1}, A_{j+1},\cdots, A_d) \mbox{ and }U(B_1, \cdots, B_{j-1}, B_{j+1},\cdots, B_d),$$ respectively.
  By our assumption, in the finite covering, each cube does not consist of  the $j$-th direction, where  $j\in\{k+1,\cdots, d\}$.
  We  translate the hypersurface $x_j=U(x_1,x_2,\cdots,x_{j-1}, x_{j+1},\cdots,x_d)$ via
  $U+c,$ where $c$ is some constant.  The parameter $c$ can be chosen in an open set.  When we translate the hypersurface  $x_j=U$,  it always lies  in the original finite covering.
  The supremum and  infimum of the parameter $c$, happen exactly  when the hypersurface  $U+c$
  reaches the boundary of  finite covering.
  So, by Lemmas \ref{finite cover} and \ref{translation}, $w$ can be chosen in an open set. For other remaining directions, the proof is similar.  We are done.
\end{proof}

\section{Applications and  Examples}
We first give the following example to illustrate that although the thicknesses of Cantor sets are small,  we still have similar consequence as the Newhouse's thickness theorem.
\begin{example}
Let $K_1$ be a self-similar set with the IFS
$$\left\{f_1(x)=\dfrac{x}{4}, f_2(x)=\dfrac{x}{4}+0.3\right\}.$$
Let $K=K_1\cup (K_1+0.6)$, where
$$K_1+0.6=\{x+0.6:x\in K_1\}.$$
$K$ is a Cantor set with convex hull $[0,1]$. Note that
$$\tau(K)=\tau(K_1)=\dfrac{1}{2}. $$
Hence, the Newhouse thickness theorem cannot be used for $K+K$.  Nonetheless, it is easy to prove that $K+K$ is also a self-similar set with the IFS
$$\left\{f_1(x)=\dfrac{x}{4},f_i(x)=\dfrac{x+0.6i}{4},2\leq i\leq 8,f_9(x)=\dfrac{x+6}{4}\right\}.$$
The convex hull of $K$ is $[0,2]$. Hence, it is easy to check
$$f_i([0,2])\cap f_{i+1}([0,2])\neq \emptyset, 1\leq i\leq 8.$$
Therefore, $$K+K=[0,2].$$
\end{example}
 We give an example to illustrate that without the linked condition we can only get that the sum of two Cantor sets is a finite union of closed intervals.
\begin{example}
Let $J=[0,1/9]\cap C$, where $C$ is the middle-third Cantor set. Then
$$J+C=[0,4/9]\cup [6/9,10/9].$$
\end{example}
 By Theorem \ref{simple case}, $$J+C=[0,4/9]\cup [6/9,10/9].$$
For the next application, we consider  the Egyptian fractions. Erd\H{o}s and Straus \cite{Erdos1980} posed the following celebrated conjecture.
\begin{conjecture}
For any positive integer $n\geq 2$ the equation
$$\dfrac{4}{n}=\dfrac{1}{x}+\dfrac{1}{y}+\dfrac{1}{z},$$
has a solution in positive integers $x, y,$ and $ z$.
\end{conjecture}
Motivated by this conjecture, we have the following result.
\begin{theorem}
Let $C$ be the middle-third Cantor set. Then
$$\dfrac{1}{C}+\dfrac{1}{C}+\dfrac{1}{C}+\dfrac{1}{C}=\left\{\dfrac{1}{x_1}+\dfrac{1}{x_2}+\dfrac{1}{x_3}+\dfrac{1}{x_4}:x_i\in C\setminus \{0\}, 1\leq i\leq 4\right\}=[4,\infty).$$
Namely, for any $a\in [4,\infty)$, there exist some $x_i\in C,1\leq i\leq 4$ such that
$$a=\dfrac{1}{x_1}+\dfrac{1}{x_2}+\dfrac{1}{x_3}+\dfrac{1}{x_4}.$$
Moreover,   our result is sharp in the sense that
$$\dfrac{1}{C}+\dfrac{1}{C}+\dfrac{1}{C} \mbox{ or }\dfrac{1}{C}+\dfrac{1}{C}$$ is not an interval.
\end{theorem}
First, we let
$f(x, y,z,w)=-\dfrac{1}{x}-\dfrac{1}{y}-\dfrac{1}{z}-\dfrac{1}{w}$. By Theorem \ref{Cantor},
$$f([2/3,1]\cap C, [2/3,1]\cap C,[2/3,1]\cap C, [2/3,1]\cap C)=[-6,-4],$$ which implies  $$\dfrac{1}{C}+\dfrac{1}{C}+\dfrac{1}{C}+\dfrac{1}{C}\supset [4,6]. $$
For any $x\in C$, we always have $x/3\in C$. This fact together with the above inclusion yield that
$$\dfrac{1}{C}+\dfrac{1}{C}+\dfrac{1}{C}+\dfrac{1}{C}\supset [12,18]. $$
Hence, to prove
$$\dfrac{1}{C}+\dfrac{1}{C}+\dfrac{1}{C}+\dfrac{1}{C}= [4,\infty)$$
it suffices to show
$$\dfrac{1}{C}+\dfrac{1}{C}+\dfrac{1}{C}+\dfrac{1}{C}\supset [6,12].$$
Similarly, by Theorem \ref{Cantor},
$$f([2/9,1/3]\cap C, [2/9,1/3]\cap C,[2/9,1/3]\cap C, [2/3,1]\cap C)=[-15,-10].$$
Therefore,
$$\dfrac{1}{C}+\dfrac{1}{C}+\dfrac{1}{C}+\dfrac{1}{C}\supset [10,15].$$
Hence, it remains to show
$$\dfrac{1}{C}+\dfrac{1}{C}+\dfrac{1}{C}+\dfrac{1}{C}\supset [6,10].$$
In terms of  Theorem \ref{Cantor},
$$f([6/27,7/27]\cap C, [8/27,1/3]\cap C,[2/3,1]\cap C, [2/3,1]\cap C)\supset [-10,-62/7].$$
As such,
$$\dfrac{1}{C}+\dfrac{1}{C}+\dfrac{1}{C}+\dfrac{1}{C}\supset [62/7,10].$$
By  Theorem \ref{Cantor},
$$f([8/27,1/3]\cap C, [8/27,1/3]\cap C,[8/9,1]\cap C, [8/9,1]\cap C)\supset [-9,-8].$$
Thus,
$$\dfrac{1}{C}+\dfrac{1}{C}+\dfrac{1}{C}+\dfrac{1}{C}\supset [8,9].$$

By means  of  Theorem \ref{Cantor},
$$f([6/27,7/27]\cap C, [2/3,1]\cap C,[2/3,1]\cap C, [2/3,1]\cap C)\supset M,$$
where
\begin{equation*}
\begin{aligned}
M&=&[-3-9/8-81/18,-81/19-1-18/7]\\
&=&f([18/81,19/81],[2/3,7/9],[2/3,7/9],[8/9,1]).
\end{aligned}
\end{equation*}
Thus,
$$\dfrac{1}{C}+\dfrac{1}{C}+\dfrac{1}{C}+\dfrac{1}{C}\supset [81/19+1+18/7,3+9/8+81/18]\approx[7.2,8.625].$$

\noindent By virtue of  Theorem \ref{Cantor} again,
$$f([8/27,1/3]\cap C, [2/3,1]\cap C,[2/3,1]\cap C, [2/3,1]\cap C)\supset [-63/8,-6].$$
Thus,
$$\dfrac{1}{C}+\dfrac{1}{C}+\dfrac{1}{C}+\dfrac{1}{C}\supset [6,63/8].$$
Therefore, we have proved
$$\dfrac{1}{C}+\dfrac{1}{C}+\dfrac{1}{C}+\dfrac{1}{C}\supset [4,12].$$
Since
$$\dfrac{1}{C}+\dfrac{1}{C}+\dfrac{1}{C}+\dfrac{1}{C}\supset \dfrac{3}{C}+\dfrac{3}{C}+\dfrac{3}{C}+\dfrac{3}{C},$$
it follows that
$$\dfrac{1}{C}+\dfrac{1}{C}+\dfrac{1}{C}+\dfrac{1}{C}=[4,\infty). $$

For the sharpness, it is easy to check that
$$\dfrac{1}{C}\subset [1,3/2]\cup [3,\infty).$$
Therefore, $$\dfrac{1}{C}+\dfrac{1}{C}+\dfrac{1}{C}\subset [3,4.5]\cup [5,+\infty),\dfrac{1}{C}+\dfrac{1}{C}\subset[2,3]\cup [4,+\infty).$$

We give the following example to consider division on some Cantor sets.
We find that the complete structure of Cantor sets under the division operation has some application to the  Diophantine equations.

It is well-known that  the Fermat's equation, i.e.
$$x^n+y^n=z^n, n\in \mathbb{N}_{\geq 3},$$ has no non-trivial solutions in
integers.
However, in the fractal setting, we can find  the solutions of the Fermat's equation on some Cantor sets.
\begin{theorem}
Let $K$ be the attractor of the following IFS
$$\{f_1(x)=\lambda x, f_2(x)=\lambda x +1-\lambda, 0<\lambda<1/2\}.$$
Let $n\in \mathbb{N}_{\geq 2}$. If   $\lambda\in \left[1/3,\dfrac{3-\sqrt{5}}{2}\right)$, then
  $$(x,y,z)\in (K\cap (0,1))\times (K\cap (0,1))\times (K\cap (0,1))$$ satisfies
$$x^n+y^n=z^n$$
 only if
$$y=(1+\alpha)x \mbox{ or }x=(1+\alpha)y, x\in  K\cap (0,1),\alpha\geq 0,$$
and there exists  $k\in \mathbb{Z}$ such that
\begin{equation*}
   \left\lbrace\begin{array}{cc}
   (\lambda^{kn}(1-\lambda)^n-1)^{1/n}-1\leq \alpha\leq (\lambda^{kn}(1-\lambda)^{-n}-1)^{1/n}-1\\
   (\lambda^{kn}(1-\lambda)^{-n}-1)^{-1/n}-1\leq \alpha\leq (\lambda^{kn}(1-\lambda)^{n}-1)^{-1/n}-1.
                \end{array}\right.
\end{equation*}
In particular, if $\lambda\in \left[1/3,\dfrac{3-\sqrt{5}}{2}\right)$, then there exist infinitely many
  $$(x,y,z)\in (K\cap (0,1))\times (K\cap (0,1))\times (K\cap (0,1))$$ such that
$$x^n+y^n=z^n.$$
\end{theorem}
\begin{proof}
Let   $(x,y,z)\in (K\cap (0,1))\times (K\cap (0,1))\times (K\cap (0,1))$ be a solution of
$$x^n+y^n=z^n.$$
We suppose without loss of generality that $y=(1+\alpha)x,   \alpha\geq 0$. Then the equation
$$x^n+y^n=z^n$$ becomes
$$\dfrac{z}{x}=\sqrt[n]{1+(1+
\alpha)^n}\in \dfrac{K}{K}.$$
Hence, to find a solution of the above equation, it suffices to consider the complete structure of
$$\dfrac{K}{K}=\left\{\dfrac{z}{x}:x,z\in K, x\neq 0\right\}. $$
However, by Theorem \ref{Cantor}, if   $\lambda\in \left[1/3,\dfrac{3-\sqrt{5}}{2}\right)$, then
$$\dfrac{K}{K}= \cup_{k=-\infty}^{\infty}[(\lambda^k(1-\lambda), \lambda^k(1-\lambda)^{-1}]\cup \{0\}.$$
Therefore, we can find the first condition in the bracket.
Similarly, if
$$y=(1+\alpha)x,$$ then
$$x=\dfrac{y}{1+\alpha}.$$ As such,
$$x^n+y^n=z^n$$ becomes
$$\dfrac{z}{y}=\sqrt[n]{1+(1+
\alpha)^{-n}}.$$
Hence, we have the second condition in the bracket.

For the second statement,  if we let  $\alpha=0$ and $x=y$, then the equation
$x^n+y^n=z^n$ becomes
$$2x^n=z^n.$$
For this case, we always have
$$\dfrac{K}{K}\supset [1-\lambda,(1-\lambda)^{-1}]\supset \{2^{1/n}\}_{n=2}^{\infty}.$$
Note that if $(x,y,z)$ is a solution of the Fermat's equation, then
$$(\lambda x, \lambda y,\lambda z)$$ is also a solution. Hence, the Fermat's equation has infinitely many solutions.
\end{proof}

Generally, once we have obtained the closed  form of $f(K_1,\cdots, K_d)$, then we can consider similar Diophantine equations on $K_1\times \cdots\times K_d$.
More precisely, if we have proved that $f(K_1,\cdots, K_d)$ is a finite union of closed intervals, i.e.
$$f(K_1,\cdots, K_d)=\cup_{i=1}^{k}I_i,$$
where $k$ is some positive integer and each $I_i$ is a closed interval,
then for any $\alpha\in (\cup_{i=1}^{k}I_i)^c$, the following equation
$$f(x_1,\cdots, x_d)=\alpha$$ does not have any solution over
$K_1\times \cdots \times K_d.$
If $\beta\in \cup_{i=1}^{k}I_i$, then the equation
$$f(x_1,\cdots, x_d)=\beta$$  has a solution on
$K_1\times \cdots \times K_d.$

Finally, we have the following  result.
\begin{theorem}
Let $K$ be the attractor of the following IFS
$$\{f_1(x)=\lambda x, f_2(x)=\lambda x +1-\lambda, 0<\lambda<1/2\}.$$
Let $n\in \mathbb{N}_{\geq 2}$. If   $\lambda\in \left[1/3,\dfrac{3-\sqrt{5}}{2}\right)$, then
$$e K\cap \pi K\neq\emptyset, e K\cap  K\neq\emptyset, \pi K\cap K\neq\emptyset, \sqrt{2} K\cap  K\neq \emptyset, $$
where $aK=\{ax:x\in K\setminus \{0\}\},a\in \mathbb{R}$.
\end{theorem}
The proof follows immediately from the following fact: if   $$\lambda\in \left[1/3,\dfrac{3-\sqrt{5}}{2}\right),$$ then
$$\dfrac{K}{K}= \cup_{k=-\infty}^{\infty}[(\lambda^k(1-\lambda), \lambda^k(1-\lambda)^{-1}]\cup \{0\}.$$

\section{Final remarks and some problems}
In this paper, we only consider the problems in $\mathbb{R}$. It is interesting to consider similar problems in $\mathbb{R}^d, d\geq 2.$
For instance,  given $E,F\subset \mathbb{R}^d$, define a continuous function $g:\mathbb{R}^{2d}\to \mathbb{R}^d$. It would be interesting to consider when $g(E,F)$ contains  some hypersurface.

 Motivated by some famous conjectures in analytic number theory, many problems can be asked. For instance, it is natural to consider the following Fermat-Catalan equation on the middle-third Cantor set.
\begin{problem}
Let $n,m,k\in \mathbb{N}_{\geq 3}$.
Then how can we find  some  non-trivial $(x,y,z)\in C \times C\times C$ satisfying
$$x^n+y^m=z^k.$$
\end{problem}
Similar problems can be asked.
\begin{problem}
Finding all possible  $$(x,y,z)\in (C\cap (0,1))\times (C\cap (0,1))\times (C\cap (0,1))$$
such that
$$x^2+y^2=z.$$
\end{problem}
\section*{Acknowledgements}
 The work is
 supported by K.C. Wong Magna Fund in Ningbo University.

\end{document}